\title{Basic Morita equivalences and a conjecture of Sasaki for cohomology of block algebras of finite groups}
\author{
        CONSTANTIN-COSMIN TODEA\\
        \emph{Technical University of Cluj-Napoca}\\
        \emph{Department of Mathematics, Str. G. Baritiu 25}\\
 \emph{Cluj-Napoca 400027, Romania}\\
 \emph{Constantin.Todea@math.utcluj.ro}
}
\date{\today}

\documentclass[10pt]{article}
\usepackage{xypic}
\usepackage{amsmath,amsthm,amssymb,amsfonts}
\newcommand{\Hom}{\mathrm{Hom}}

\newcommand{\Id}{\mathrm{Id}}
\newcommand{\op}{\mathrm{op}}

\newcommand{\Res}{\mathrm{Res}}

\newcommand{\tr}{\mathrm{tr}}
\newcommand{\End}{\mathrm{End}}
\newcommand{\Ext}{\mathrm{Ext}}
\newcommand{\Ind}{\mathrm{Ind}}

\newcommand{\BaR}{\mathrm{Bar}}
\newcommand{\HH}{\mathrm{HH}}
\newcommand{\Hc}{\mathrm{H}}

\newcommand{\Ker}{\mathrm{Ker}}
\newcommand{\Ima}{\mathrm{Im}}

\newcommand{\res}{\mathrm{res}}
\newcommand{\Br}{\mathrm{Br}}
\newtheorem{thm}{Theorem}[section]
\newtheorem{theorem}[thm]{Theorem}
\newtheorem{corollary}[thm]{Corollary}
\newtheorem{lemma}[thm]{Lemma}
\newtheorem{proposition}[thm]{Proposition}
\newtheorem{definition}[thm]{Definition}
\newtheorem{nim}[thm]{}
\newtheorem{remark}[thm]{Remark}

\begin{document}
\maketitle


\begin{abstract}
We show, using basic Morita equivalences between block algebras of finite groups, that the Conjecture of H. Sasaki from \cite{Sas} is true for a new class of blocks called nilpotent covered blocks. When this Conjecture is true we define some transfer maps between cohomology of block algebras and we analyze them through transfer maps between Hochschild cohomology algebras.
\end{abstract}

\section{Introduction}\label{sec1}

Let $G$ be a finite group and let $k$ be an
algebraically closed  field of prime characteristic $p>0$ dividing the order of $G$. Let $kGb$ be a block algebra where $b$ is the primitive idempotent in the center of $kG$, $b\in Z(kG)$, with defect group $P$. Let $P_{\gamma}$ be a defect pointed group of $G_{\{b\}}$, $(P,e_P)$ be a maximal $b$-Brauer pair associated to $P_{\gamma}$ and let $i\in \gamma$ be a source idempotent. For any subgroup $R$ of $P$ there is a unique block $e_R$ of $C_G(R)$ such that $\Br_R(i)e_R \neq 0$. Then $(R, e_R)$ is a $b$-Brauer pair and $e_R$ is also the unique block of $C_G(R)$ such that $(R, e_R) \leq (P, e_P)$. We define $\mathcal{F}_{(P,e_P)}(G,b)$ as the  category which has as objects the set of subgroups of $P$; for any two subgroups $R,S$ of $P$ the set of morphisms from $R$ to $S$ in  $\mathcal{F}_{(P,e_P)}(G,b)$ is the set of (necessarily injective) group homomorphisms $\varphi:R\rightarrow S$ for which there is an element $x \in G$ satisfying $\varphi(u)=xux^{-1}$ for all $u\in R$ and satisfying $^x(R, e_R)\leq (S, e_S)$. The category $\mathcal{F}_{(P,e_P)}(G,b)$ is equivalent to the Brauer category of $b$ with respect to the choice of $(P,e_P)$ and is a saturated fusion system. More than fifteen years ago M. Linckelmann defines in \cite{LiTr} the cohomology algebra of the block $b$, denoted $\Hc^*(G,b,P_{\gamma})$ which can  also be viewed now as the cohomology algebra of a fusion system $\Hc^*(\mathcal{F}_{(P,e_P)}(G,b))$, see \cite{BLO} for cohomology of fusion systems. The $P$-interior algebra $ikGi$ is called the source algebra of $kGb$. As $kP-kP$-bimodule $ikGi$ is a direct summand of $kG$, hence there is $Y_{G,b,P_{\gamma}}\subseteq [P\backslash G /P]$ such that $$ikGi\cong \bigoplus_{g\in Y_{G,b,P_{\gamma}}}k[PgP],$$
as $kP-kP$-bimodules, where $[P\backslash G/P]$ is a system of representatives of double cosets of $P$ in $G$. The set $Y_{G,b,P_{\gamma}}$ has the property that  contains  a system of representatives $[N_G(P_{\gamma})/PC_G(P)]$ for any block $b$.

In \cite{Sas} H. Sasaki defines a linear map as follows
$$t_{Y_{G,b,P_{\gamma}}}:\Hc^*(P,k)\rightarrow\Hc^*(P,k); ~~~~~~~[\zeta]\mapsto\sum_{g\in Y_{G,b,P_{\gamma}}}\tr_{P\cap^gP}^P\res_{P\cap^gP}^{^gP}{}^g[\zeta],$$
and he proposes the following conjecture.

\textbf{Conjecture.} With the above notations it follows that
$$\Hc^*(G,b,P_{\gamma})=t_{Y_{G,b,P_{\gamma}}}(\Hc^*(P,k)).$$
For shortness we also use the notation $$t_{Y_{G,b,P_{\gamma}}}=\sum_{g\in Y_{G,b,P_{\gamma}}}t_g$$ where for $g \in Y_{G,b,P_{\gamma}}$ the graded linear map $t_g$ is defined by
 $$t_g:\Hc^*(P,k)\rightarrow\Hc^*(P,k),~~~~~~~[\zeta]\mapsto\tr_{P\cap^gP}^P\res_{P\cap^gP}^{^gP}{}^g[\zeta].$$
Since we will we work with more than one block algebra we are forced  to use a little cumbersome notation for $t_{Y_{G,b,P_{\gamma}}}$ as opposed to the original one from \cite{Sas}. Broto  et al. \cite{BLO} show that for any saturated fusion system on $P$ there is a $P-P$-biset that induces a linear map analogous to $t_{ Y_{G,b,P_{\gamma}}}$ whose image is the cohomology of the saturated fusion system; such a biset is called \emph{characteristic biset}. For the convenience of the reader we begin with a well-known remark giving a class of blocks for which the above Conjecture is true. The arguments for the proof of this remark are the same with the arguments from \cite[Example 2]{Sas}. For completeness we point out that  we denote by $\mathcal{F}_P(N_G(P_{\gamma}))$ the fusion system on $P$ in $N_G(P_{\gamma})$ (not necessarily saturated) with objects the subgroups of $P$ and morphisms given by conjugation in $N_{G}(P_{\gamma})$; we use $^gR=gRg^{-1}$ where $R$ is a subgroup of $P$ and $g\in G$.

\begin{remark}\label{remcontrolnormalizer}With the above notations we assume that $b$ is a block for which the inertial group $N_G(P_{\gamma})$ controls the fusion of $b$-Brauer pairs, $\mathcal{F}_{(P,e_P)}(G,b)\subseteq \mathcal{F}_{P}(N_G(P_{\gamma}))$. Then the above Conjecture is true.
In particular, we have the following list of blocks for which the inertial group controls the fusion of $b$-Brauer pairs:
\begin{itemize}
  \item $P$ is normal in $G$; this is \cite[Example 1]{Sas}.
  \item $P$ is abelian; this is \cite[Example 2]{Sas}.
  \item $P$ is a $p$-group such that the hyperfocal subgroup of $P_{\gamma}$ is a subgroup in the center of $P$ (in particular this happens when $P$ is abelian); this follows from the proof of \cite[Theorem 1,(ii)]{Wat}. Recall that by \cite[1.7]{PuHyp} the subgroup $\langle[U,O^p(N_G(U,e_U))]\mid U\leq P\rangle$ is called the \emph{hyperfocal subgroup} of $P_{\gamma}$, where $O^p(N_G(U,e_U))$ is the unique smallest normal subgroup of $N_G(U,e_U)$ of $p$-power index.
\end{itemize}
\end{remark}
In the main theorem of this article we will show how basic Morita equivalent blocks give new information regarding the image of the above linear map in the Conjecture, if a condition is satisfied. This theorem provides a  method for proving Sasaki's Conjecture for the so-called nilpotent covered blocks. Let $b$ be a block with the above notations and let $H$ be a different finite group. Let $c$ be a block of $kH$ with defect pointed group $Q_{\delta}$. The similar notations and definitions are considered for $c$: a source idempotent is $j\in \delta$, the interior $Q$-algebra $jkHj$ is the source algebra, the Brauer category is $\mathcal{F}_{(Q,f_Q)}(H,c)$, the cohomology algebra is $\Hc^*(H,c,Q_{\delta})$, etc. Let $M$ be a $kG-kH$-bimodule such that $bMc=M$ and $M$ is indecomposable considered as $k[G\times H^{op}]$-module. We take $P'\leq G\times H$ to be the vertex of $M$ and $N$ is a $kP'$-indecomposable source of $M$. We say that $M$ defines a \emph{basic Morita equivalence} between $b$ and $c$ if $M$ gives a Morita equivalence such that the equivalent conditions of \cite[Corollary 7.4]{PuBook} are satisfied. For example this happens if $p$ does not divides the $k$-dimension  of $N$ or if $S=\End_k(N)$ is a Dade $P'$-algebra.
\begin{theorem}\label{thm12} Let $b,c$ be two blocks with the above notations. We assume that $b$ is basic Morita equivalent to $c$. If $Y_{H,c,Q_{\delta}}=[N_H(Q_{\delta})/QC_H(Q)]$  then $$\Ima t_{Y_{G,b,P_{\gamma}}}\subseteq \Hc^*(P,k)^{N_G(P_\gamma)/PC_G(P)}.$$
\end{theorem}
A block $c$ with  normal defect group in $H$ satisfies the condition $$Y_{H,c,Q_{\delta}}=[N_H(Q_{\delta})/QC_H(Q)]$$  from  Theorem \ref{thm12}. In particular the same arguments as for the proof of Remark \ref{remcontrolnormalizer} assure us that for these blocks the Conjecture of Sasaki is true.
We say that  $b$ is a \emph{nilpotent covered} block if there is $\widetilde{G}$ a finite group such that $G$ is a normal subgroup of $\widetilde{G}$ and there is $\widetilde{b}$ a nilpotent block of $k\widetilde{G}$ which covers $b$. We know that $e_P$ is lifted to a block $\widehat{e_P}$ of the inertial group with the same defect $P$ and \cite[Corrollary 4.3]{PuNilext} assure us  that a nilpotent covered block is basic Morita equivalent to $\widehat{e_P}$ as block of $N_G(P_{\gamma})$ which has the same defect pointed group $P_{\gamma}$. Since $b$ is basic Morita equivalent to $\widehat{e_P}$ we know from \cite{PuBook} that $\mathcal{F}_{(P,e_P)}(G,b)$  is isomorphic to $\mathcal{F}_{(P,e_P)}(N_G(P_{\gamma}),\widehat{e_P})$ as finite categories hence by Theorem \ref{thm12} we obtain
 $$ \Ima t_{Y_{G,b,P_{\gamma}}}\subseteq\Hc^*(P,k)^{N_G(P_{\gamma})/PC_G(P)}=\Hc^*(N_G(P_{\gamma}),\widehat{e_P},P_{\gamma})\cong \Hc^*(G,b,P_{\gamma}).$$
Since the last isomorphism is an isomorphism of graded vector spaces of finite type we obtain the next corollary.
\begin{corollary}\label{cornilpext}
For any nilpotent covered block the Conjecture of Sasaki is true.
\end{corollary}
It might be possible to obtain a shorter proof for the above corollary since basic Morita equivalence preserves the Brauer categories and one of the blocks has normal defect group but we prefer the use of Theorem \ref{thm12} to emphasize a new method.
Section \ref{sec2} is devoted for the proof of this theorem. The aim of Section \ref{sec3} is to fill a gap in the literature, since for experts the fact that Hochschild cohomology of group algebras is a Mackey functor is known. So it is worth spending some time to analyze and to give an explicit approach for the maps defining the Mackey functor in Section \ref{sec3}, using bar resolution. The only new result is Proposition \ref{propcompconjdelta} which is used  in the proof of Theorem \ref{thmcomptran}, theorem which will be exposed below.
We propose an approach for defining  some transfer maps between cohomology of block algebras, maps which we define below. For this definition we need the existence of the map $t_{ Y_{G,b,P_{\gamma}}}$. In Section \ref{sec4} we will study these maps. Let $b,c$ be blocks of $kG$, respectively  $kH$ with defect groups $P$ respectively $Q$ such that $Q\leq P$ and $H\leq G$. We take $P_{\gamma},$ respectively $Q_{\delta}$ as defect pointed groups and we denote by $(Q,f_Q)$ a maximal $(H,c)$-Brauer pair associated to $Q_{\delta}$. Again $j\in \delta$ is a source idempotent for $c$, the fusion system associated to $c$ is $\mathcal{F}_{(Q,f_Q)}(H,c)$  and $jkHj$ is the source algebra. So, if the Conjecture is true for $b$  we can complete the following commutative diagram
\begin{displaymath}
 \xymatrix{\Hc^*(H,c,Q_{\delta})\ar@{.>}[rr]^{ \tr^b_c}\ar@{^{(}->}[d] && \Hc^*(G,b,P_{\gamma}) \ar@{<<-}[d]^{t_{Y_{G,b,P_{\gamma}}}} \\
                                                     \Hc^*(Q,k)\ar[rr]^{\tr_Q^P}
                                                     &&\Hc^*(P,k)
                                            }.
\end{displaymath}
\begin{definition}\label{defntrbc} With the above notations we assume that the Conjecture is true for $b$. We define the graded $k$-linear map
$$\tr^b_c:\Hc^*(H,c,Q_{\delta})\rightarrow\Hc^*(G,b,P_{\gamma});~~~~~~~[\zeta]\mapsto \tr^b_c([\zeta])=t_{Y_{G,b,P_{\gamma}}}(\tr_Q^P([\zeta])).$$
Explicitly $\tr^b_c([\zeta])=\sum_{g\in Y_{G,b,P_{\gamma}}}\tr^P_{P\cap^gP}\res^{^gP}_{P\cap^gP}{}^g(\tr^P_Q([\zeta]))$. We call $\tr^b_c$ a transfer map in cohomology of block algebras.
\end{definition}
Some ideas for defining these transfer maps are suggested in \cite{SasII}. For defining  a restriction map between cohomology of block algebras, we need to assume that
$\mathcal{F}_{(Q,f_Q)}(H,c)$ is a subsystem of $\mathcal{F}_{(P,e_P)}(G,b)$. In Section \ref{sec4} and in the rest of this Introduction we work with blocks satisfying this condition. This inclusion and $\res^P_Q$ induce a well-defined  restriction map between cohomology algebras of blocks, denoted
$$\res^b_c:\Hc^*(G,b,P_{\gamma})\rightarrow\Hc^*(H,c,Q_{\delta}); ~~~~~~~[\zeta]\mapsto \res^b_c([\zeta])=\res^P_Q([\zeta]).$$
In order that the definition of the transfer map $\tr_c^b$ to be reasonable we need that good properties such as Frobenius reciprocity and transitivity law $$\tr_c^b\circ t_{Y_{H,c,P_{\delta}}}=t_{Y_{G,b,P_{\gamma}}}$$ to hold. We prove Frobenius reciprocity in Proposition \ref{propreciprtransf} but the law of transitivity is proved  in Proposition \ref{proptranslaws} for $b$ and $c$ blocks such that $$\Ker ~t_{Y_{H,c,P_{\delta}}}\subseteq \Ker~t_{Y_{G,b,P_{\gamma}}}.$$ In the next theorem, which will be proved in Section \ref{sec4}, we try to understand transfer maps in block cohomology via transfer maps between Hochschild cohomology of block algebras, which are in Brauer correspondence; therefore from now we take $Q=P$. For the rest of Section \ref{sec4} and of this Introduction  let $H$ be a subgroup of $G$ containing $PC_G(P)$ where $P$ is the common defect group of $b$ and $c$, where $c$ is the Brauer correspondent of $b$. We choose $i\in \gamma$ and $j\in \delta$ such that the source modules $A=kGi, B=kHj$ are in Green correspondence with respect to $(G\times P^{op},\Delta P, H\times P^{op})$. We denote by $L=L(kGb,kHc)$ the Green correspondent of $kHc$, viewed as an indecomposable $kH-kH$-bimodule with respect to $(G\times H^{op},\Delta P, H\times H^{op})$. In \cite{LiTr} Linckelmann defines the subalgebra of stable elements in the Hochschild cohomology algebra of a symmetric algebra and transfer maps between Hochschild cohomology of symmetric algebras. For example, since $ikGi$ is a $kP-kP$-bimodule, projective as left and right module, there is a transfer map associated to $ikGi$ denoted $t_{ikGi}:\HH^*(kP)\rightarrow\HH^*(kP)$. We denote by $T$  the transfer $t_{ikGi}$ restricted to $\HH^*_{B^*}(kP)\cap\delta_P(\Hc^*(P,k))$ (which is isomorphic to $\Hc^*(H,c,P_{\delta})$ by \cite[Theorem 1]{Sas}). Recall that $\HH^*_{B^*}(kP)$ is the subalgebra of $B^*$-stable elements; that is $[\zeta]\in \HH^*_{B^*}(kP)$ if there is $[\theta]\in\HH^*(kHc)$ such that $[\zeta]\otimes_{kP}[\Id_{B^*}]=[\Id_{B^*}]\otimes_{kHc} [\theta]$ in $\Ext_{kP\otimes_k (kHc)^{op}}(B^*,B^*)$, see \cite[Definition 2]{Sas}.
\begin{theorem}\label{thmcomptran} With the above conditions we assume that $\mathcal{F}_{(P,f_P)}(H,c)\subseteq \mathcal{F}_{(P,e_P)}(G,b)$. If we assume that the Conjecture is true for $b$ then the following diagram of graded $k$-linear maps is commutative
\begin{displaymath}
 \xymatrix{\Hc^*(H,c,P_{\delta})\ar[rr]^{ R_B\circ T\circ\delta_P}\ar[d]^{\tr^b_c} && \HH^*_{L^*}(kHc) \ar@{>->>}[d]^{R_L} \\
                                                     \Hc^*(G,b,P_{\gamma})\ar[rr]^{R_A\circ\delta_P}
                                                     &&\HH^*_{L}(kHc)
                                            }.
\end{displaymath}
\end{theorem}
The homomorphisms of graded $k$-algebras $R_A,R_B$ and $R_L$ are obtained by restricting the normalized transfer maps $T_A,T_B$ and $T_L$, respectively, to some subalgebras of stable elements, see \cite[Theorem 3.6]{LiTr}. For basic facts in block theory we refer to \cite{TH} and for properties and notations regarding transfer maps in Hochschild cohomology we refer to \cite{LiTr} and \cite{Sas}.

\section{Proof of Theorem \ref{thm12}}\label{sec2}
We denote by $C~\bigg|~D$ the fact that $C$ is isomorphic to a direct summand of $D$, where $C,D$ are some $kG$-modules.
\begin{proof}
Since $b$ is basic Morita equivalent to $c$, by \cite[Corollary 7.4]{PuBook} we can identify $P=Q$, hence $P_{\gamma}\leq G_{\{b\}}$ is a defect pointed group with $i\in\gamma$  and $P_{\delta}\leq H_{\{c\}}$ is a defect pointed group with $j\in\delta.$ By \cite[Theorem 6.9]{PuBook} we know that we can construct two embeddings of $P$-interior algebras
$$ikGi\rightarrow S\otimes jkHj$$
and $$jkHj\rightarrow S^{\op}\otimes ikGi.$$
In particular we obtain
\begin{equation}\label{eq1} ikGi~\bigg|~S\otimes jkHj
\end{equation}
as $k[P\times P^{\op}]$-module.
Let $x\in Y_{G,b,P_{\gamma}}$ then $k[PxP]$ is an indecomposable direct summand of $ikGi$ as $k[P\times P^{\op}]$-module, hence by (\ref{eq1}) we get that
$$k[PxP]~\bigg|~\bigoplus_{y\in Y_{H,c,P_{\delta}}}(S\otimes k[PyP])$$
as $k[P\times P^{\op}]$-module. It follows that there is $y\in Y_{H,c,P_{\delta}}$ such that $$k[PxP]~\bigg|~S\otimes k[PyP]$$ as $k[P\times P^{\op}]$-module. It is known that $k[PyP]$ is a trivial source module of vertex ${}^{(y,1)}\Delta(^{y^{-1}}P\cap P)$ which satisfies
$$k[PyP]~\bigg|~\Ind_{{}^{(y,1)}\Delta(^{y^{-1}}P\cap P)}^{P\times P^{\op}}k;$$ the similar statements are true for $k[PxP]$. It follows that $$k[PxP]~\bigg|~S\otimes \Ind_{{}^{(y,1)}\Delta(^{y^{-1}}P\cap P)}^{P\times P^{\op}}k$$ as $k[P\times P^{\op}]$ and since we have the $k[P\times P^{\op}]$-modules isomorphism  $$S\otimes \Ind_{{}^{(y,1)}\Delta(^{y^{-1}}P\cap P)}^{P\times P^{\op}}k\cong\Ind_{{}^{(y,1)}\Delta(^{y^{-1}}P\cap P)}^{P\times P^{\op}}(\Res_{{}^{(y,1)}\Delta(^{y^{-1}}P\cap P)}^{P\times P^{\op}}(S))$$
 we get that $k[PxP]$ is ${}^{(y,1)}\Delta(^{y^{-1}}P\cap P)$-projective. But its vertex is $${}^{(x,1)}\Delta(^{x^{-1}}P\cap P)$$ so we obtain that there is $(a,b)\in P\times P^{\op}$ such that
 $${}^{(x,1)}\Delta(^{x^{-1}}P\cap P)\leq{}^{(a,b)(y,1)}\Delta(^{y^{-1}}P\cap P).$$

It follows that for any $s\in {}^{x^{-1}}P\cap P$ there is a unique elements $t\in{}^{y^{-1}}P\cap P$ such that
$${}^xs={}^{ay}t,~~~s^{-1}=(b\cdot t\cdot b^{-1})^{-1}$$
where $b\cdot t\cdot b^{-1}=b^{-1}tb$ in $P^{\op}$, thus
$${}^xs={}^{ay}t,~~~s^{-1}={}^{b^{-1}}t^{-1}.$$
We obtain ${}^xs={}^{ayb}s$ for any $s\in{}^{x^{-1}}P\cap P$, hence there is $c\in C_G(P\cap{}^xP)$ such that
\begin{equation}\label{eq3}
ayb=cx.
\end{equation}
Let $t_x([\zeta])\in \Ima t_x$ for some $[\zeta]\in\Hc^*(P,k)$. Then we have
\begin{align*}
t_x([\zeta])&=\tr_{P\cap^x P}^P\res^{{}^xP}_{P\cap{}^xP}{}^x[\zeta]\\&=\tr_{P\cap^x P}^P{}^c(\res^{{}^xP}_{P\cap{}^xP}{}^x[\zeta])~~~~~~~~~~~~~~~~~(\because c\in C_G(P\cap{}^xP))\\&=\tr_{P\cap^x P}^P{}\res^{{}^{cx}P}_{P\cap{}^xP}{}^{cx}[\zeta])\\&=\tr_{P\cap^x P}^P{}\res^{{}^{ayb}P}_{P\cap{}^xP}{}^{ayb}[\zeta])~~~~~~~~~~~~~~~~~(\because (\ref{eq3}))\\&=\tr_{P\cap^x P}^P{}^a(\res^{{}^{y}P}_{{}^{a^-1}(P\cap{}^xP)}{}^y[\zeta])~~~~~~~~~~~(\because b\in P)\\&=\tr_{{}^{a^{-1}}(P\cap^x P)}^P{}(\res^{{}^{y}P}_{{}^{a^-1}(P\cap{}^xP)}{}^y[\zeta])~~~~~~~(\because a\in P)\\&=|P:{}^{a^{-1}}(P\cap{}^xP)|t_y([\zeta])~~~~~~~~~~~~~~~(\because y\in N_H(P_{\delta}))
{}\end{align*}
Since $P={}^{a^{-1}}(P\cap{}^xP)$  if and only if $x\in N_G(P_{\gamma})$ we obtain that
 $$t_x([\zeta])=t_y([\zeta])$$ for some $y$ which depends on $x$ if $x\in N_G(P_{\gamma})$ and $t_x([\zeta])$ is $0$ if $x\notin  N_G(P_{\gamma})$. It follows that for any $[\eta]\in\Ima t_{Y_{G,b,P_{\gamma}}}$ there is $[\zeta]\in\Hc^*(P,k)$ such that
 $$[\eta]=\sum_{x\in Y_{G,b,P_{\gamma}}}|P:{}^{a^{-1}}(P\cap{}^xP)|t_y([\zeta])=\sum_{x\in[N_G(P_{\gamma})/PC_G(P)]}t_x([\zeta])$$
 $$=\tr_{PC_G(P)}^{N_G(P_{\gamma})}[\zeta],$$
 which is what we need.
\end{proof}
\section{Hochschild cohomology of group algebras as Mackey functor}\label{sec3}
In this section we use transfer maps between Hochschild cohomology algebras of symmetric algebras, defined originally in \cite{LiTr}, to give a structure of Mackey functor for Hochschild cohomology of group algebras. We will work in this section under the general assumption that $k$ is a commutative ring. Let $H$ be a subgroup of $G$ and let $g\in G$. For shortness we denote by $_HkG_G$ the $kH-kG$-bimodule structure on $kG$ given by multiplication in $G$, with its $k$-dual $_GkG_H$. Also we consider $_{{}^gH}k[gH]_H$ to be the $(k[^gH]-kH)$-bimodule given by
$$gh_1g^{-1}(gh)h_2=gh_1hh_2,$$
for any $h_1,h_2,h\in H$.
The following notations are used
$$r_H^G=t_{_HkG_G},~~t_H^G=t_{_GkG_H},~~c_{g,H}=t_{_{{}^gH}k[gH]_H}, $$
hence we have the following graded $k$-linear maps
$$r_H^G:\HH^*(kG)\rightarrow\HH^*(kH);$$
$$t_H^G:\HH^*(kH)\rightarrow \HH^*(kG);$$
$$c_{g,H}:\HH^*(kH)\rightarrow\HH^*(k[^gH]);$$
which can be viewed as: restriction, transfer and conjugation maps.
The next well-known lemma is straightforward.
\begin{lemma}\label{lemkgh} Let $K,H\leq G$ and $g\in G$. Then
$$kK\otimes_{k[K\cap{}^gH]}k[gH]\cong k[KgH]$$
as $(kK-kH)$-bimodules.
\end{lemma}
In the next proposition we verify that the next quadruple is a Mackey functor, see \cite[\S 53]{TH}:$$(\HH^*(kH),r^G_H,t^G_H,c_{g,H})_{H\leq G,g\in G}.$$
\begin{proposition} Let $K\leq H\leq G$ and $g,h\in G$. The following statements are true.
\begin{itemize}
\item[i)] $r^H_K\circ r^G_H=r^G_K,~t^G_H\circ t^H_K=t^G_K;$
\item[ii)] $r^H_H=t^H_H=\Id_{\HH^*(kH)};$
\item[iii)] $c_{gh,H}=c_{g,^hH}\circ c_{h,H};$
\item[iv)] $c_{h,H}=\Id_{\HH^*(kH)}$ if $h\in H$;
\item[v)] $c_{g,K}\circ r^H_K=r^{{}^gH}_{{}^gK}\circ c_{g,H}$ and $c_{g,H}\circ t^H_K=t^{{}^gH}_{{}^gK}\circ c_{g,K};$
\item[vi)] $r^G_K\circ t^G_H=\sum_{g\in[K\backslash G/ H]}t_{K\cap {^g}H}^K\circ r^{{}^gH}_{K\cap {}^gH}\circ c_{g,H,~~}$ where $[K\backslash G/H]$ is a system of representatives of double cosets $KgH$ with $g\in G$.
\end{itemize}
\end{proposition}
\begin{proof}
Statements ii) and iv) are an easy exercise if we use \cite[Proposition 2.7 (4)]{LiZh}. The second part of statements (i) and (v) are analogous to the first part and are left as an exercise. The rest of the proof is a consequence of \cite[Proposition 2.7 (1)]{LiZh} and some bimodule isomorphisms.
For statement(i) we have
$$r^H_K\circ r^G_H=t_{_KkH_H}\circ t_{_HkG_G}=t_{_KkH\otimes_{kH}kG_G}=r^G_K.$$
For statement iii) we have
$$c_{gh,H}=t_{_{^{gh}H}k[ghH]_H};$$
$$c_{g,^hH}\circ c_{h,H}=t_{_{^{gh}H}k[g(^hH)]_{^hH}}\circ t_{_{^hH}k[hH]_H}=t_{_{^{gh}H}k[g(^hH)]\otimes_{k[^hH]}k[hH]_H};$$
Since $$k[g(^hH)]\otimes_{k[^hH]}k[hH]\cong k[ghH], ~~ghxh^{-1}\otimes hy\mapsto ghxy,$$
for any $x,y\in H$, is an isomorphism of $(k[^{gh}H]-kH)$-bimodules, we obtain statement iii). Statement v) is similar to statement iii) and is left for the reader.
For statement vi) we have
$$r^G_K\circ t^G_H=t_{_KkG_G}\circ t_{_GkG_H}=t_{_KkG_H};$$

\begin{align*}
\sum_{g\in[K\backslash G/H]}t_{K\cap^gH}^K\circ r_{K\cap^gH}^{^gH}\circ c_{g,H}&=\sum_{g\in[K\backslash G/H]}t_{_KkK\otimes_{k[K\cap^gH]}k[^gH]\otimes_{k[^gH]}k[gH]_H}\\&=\sum_{g\in[K\backslash G/H]}t_{_Kk[KgH]_H}\\&=t_{_KkG_H},
\end{align*}
where the second equality holds by Lemma \ref{lemkgh} and the last equality is given  by \cite[Proposition 2.11]{LiTr}.
\end{proof}
\begin{remark} The above proposition is generalized in \cite{CoTo} from group algebras to some fully group-graded algebras which are symmetric $k$-algebras and satisfies a condition which states that all fully group-graded subalgebras with respect to subgroups must be parabolic subalgebras.
\end{remark}
In \cite[Proposition 2.5]{LiZh} the authors give explicit definitions for transfer maps between Hochschild cohomology algebras of symmetric algebras in terms of chain maps using the bar resolution. It is worth spending some time in giving these technical definitions for the above maps. For each of the following map we find a suitable set of generators and some maps (denoted $\varphi_i$ in \cite{LiZh}) as in \cite[Proposition 2.5]{LiZh}, but we omit these details which are left for the reader. Let $n\geq 0$ be an integer and denote  by $\BaR_*(kG)$ the bar resolution of $kG$ as $(kG)^e$-module. Then $(C^*(kG),d)$ is the Hochschild cochain complex  where $$C^n(kG)=\Hom_{(kG)^e}(\BaR_n(kG),kG)\cong \Hom_k((kG)^{\otimes n},kG),$$ and $d$ are differentials induced by the bar resolution. It is well known that $$\HH^n(kG)=\Hc^n(C^*(kG),d).$$ Remember that $kG,kH,k[^gH]$ are symmetric $k$-algebras and we denote by $s_G,s_H,s_{^gH}$ their symmetrizing forms.
\begin{nim}\label{resH}Restriction map:\end{nim} $$r_H^G:\HH^*(kG)\rightarrow\HH^*(kH),~~[f]\mapsto r_H^G([f])=[r_H^G(f)],$$
where $$r_H^G(f)(h_1\otimes\ldots \otimes h_n)=\sum_{h\in H}s_G(h^{-1}f(h_1\otimes\ldots \otimes h_n))h,$$ for  $h_1,\ldots, h_n\in H$ and $f\in C^n(kG)$.
\begin{nim}\label{conjH}Conjugation map:\end{nim} $$c_{g,H}:\HH^*(kH)\rightarrow\HH^*(k[^gH]),~~[f]\mapsto c_{g,H}([f])=[c_{g,H}(f)],$$
where $$c_{g,H}(f)(^gh_1\otimes\ldots\otimes^gh_n)=\sum_{h\in H}s_H(h^{-1}f(h_1\otimes\ldots\otimes h_n))ghg^{-1}$$$$=gf(h_1\otimes\ldots \otimes h_n)g^{-1},$$
for $h_1,\ldots,h_n\in H$ and $f\in C^n(kH)$.

\begin{nim}\label{trH}Transfer map:\end{nim}
Let $\{x_i\}_{1\leq i\leq l}$ be a system of representatives of left cosets of $H$ in $G$. Let $\varphi_i:kG\rightarrow kH$ be the right $kH$-module homomorphisms defined by $\varphi_i(g)=h$ if $g=x_ih\in x_iH$, and $\varphi_i(g)=0$ if $g\notin x_iH$, where $x_i$ such that $1\leq i\leq l$.
$$t_H^G:\HH^*(kH)\rightarrow\HH^*(kG),~~[f]\mapsto t_H^G([f])=[t_H^G(f)],$$
where $$t_H^G(f)(g_1\otimes\ldots \otimes g_n)=\sum_{1\leq i_0,i_1,\ldots,i_n\leq l}x_{i_1}f(\varphi_{i_1}(g_1x_{i_0})\otimes\ldots \otimes \varphi_{i_n}(g_nx_{i_0}))x_{i_0}^{-1},$$ for  $g_1,\ldots, g_n\in G$ and $f\in C^n(kH)$.

In \cite{LiTr} Linckelmann gives the diagonal induction $\delta_G:\Hc^*(G,k)\rightarrow \HH^*(kG)$ as an injective graded $k$-algebra
 homomorphisms in language of chain homotopy maps. We prefer to use the bar resolution for group cohomology. Recall that $\Hc^n(G,k)=\Hc^n(C^*(G,k),\partial^*)$ where $$C^n(G,k)=\{f\mid f:G^n\rightarrow k\textrm{ is a set map}\},$$
 and $\partial^n:C^n(G,k)\rightarrow C^{n+1}(G,k)$ is defined for any $g_1,\ldots, g_{n+1}\in G$ by
 $$\partial^n(f)(g_1,\ldots, g_{n+1})=$$$$g_1f(g_2,\ldots,g_{n+1})+\sum_{i=1}^n(-1)^if(g_1,\ldots,g_ig_{i+1},\dots,g_{n+1})+(-1)^{n+1}f(g_1,\ldots,g_n).$$
\begin{nim}\label{diagH}Diagonal induction:\end{nim}$$\delta_G:\Hc^*(G,k)\rightarrow\HH^*(kG),~~[f]\mapsto\delta_G([f])=[\delta_G(f)],$$
where $$\delta_G(f)\in C^n(kG),~\delta_G(f)(g_1\otimes\ldots \otimes g_n)=f(g_1,\ldots,g_n)g_1\ldots g_n,$$
 for $f\in C^n(G,k)$ and $g_1, \ldots, g_n\in G$. Notice that $\delta_G$ can be obtained in these terms as $\delta_G=\alpha^*\circ \theta_{1_G}^*$, where $\theta_{1_G}^*:\Hc^*(G,k)\rightarrow\Hc^*(G,kG)$ is defined in  \cite[p.138]{SiWh} and $\alpha^*:\Hc^*(G,kG)\rightarrow\HH^*(kG)$ is the isomorphism obtained from $$\alpha^*:C^*(G,kG)\rightarrow C^*(kG,kG)$$ (by abuse of notation) given by
 $$f\mapsto \alpha^n(f),~~~\alpha^n(f)(g_1\otimes\ldots\otimes g_n)=f(g_1,\ldots,g_n)g_1\ldots g_n,$$
 for $f\in C^n(G,kG)$ and $\alpha^n(f)\in C^n(kG,kG)$.

By \cite[Proposition 25.1]{We} the conjugation map in cohomology of groups can be defined with bar resolution explicitly.
\begin{nim}\label{conjg}Conjugation map for group cohomology:\end{nim}
$$^g\cdot:\Hc^*(H,k)\rightarrow\Hc^*(^gH,k),~~[f]\mapsto^g[f]=[^gf],$$
where $$^gf\in C^n(^gH,k),~^gf(^gh_1,\ldots,^gh_n)=gf(h_1,\ldots,h_n)=f(h_1,\ldots,h_n),$$
for $f\in C^n(H,k)$ and $h_1,\ldots,h_n\in H$.

\begin{proposition}\label{propcompconjdelta} With the above notations the following diagram is commutative
\begin{displaymath}
 \xymatrix{\Hc^*(H,k)\ar[rr]^{ ^g\cdot}\ar[d]^{\delta_H} && \Hc^*(^gH,k) \ar[d]^{\delta_{{}^gH}} \\
                                                     \HH^*(kH)\ar[rr]^{c_{g,H}}
                                                     &&\HH^*(k[^gH])
                                            }.
\end{displaymath}
\end{proposition}

\begin{proof} Let $ f\in C^n(H,k), ^gh_1,\ldots,^gh_n\in{}^gH$. The following hold
\begin{align*}
c_{g,H}(\delta_H(f))(^gh_1\otimes\ldots\otimes {}^gh_n)&=g\delta_H(f)(h_1\otimes\ldots\otimes h_n)g^{-1}~~~~~~~~~~~~~(\because \ref{conjH})\\
&=gf(h_1,\ldots,h_n)h_1\ldots h_ng^{-1}~~~~~~~~~~~~(\because \ref{diagH})\\
&=f(h_1,\ldots,h_n)~{}^g(h_1\ldots h_n);
\end{align*}
\begin{align*}
\delta_{^gH}(^gf)(^gh_1\otimes\ldots\otimes {}^gh_n)&=(^gf)(^gh_1,\ldots,^gh_n)~{}^gh_1\ldots^gh_n ~~~~~~~~~~~~~~~~(\because\ref{diagH})\\ &=(gf(h_1,\ldots,h_n))~{}^g(h_1\ldots h_n)~~~~~~~~~~~~~~~~~~(\because\ref{conjg})
\\ &=f(h_1,\ldots,h_n)~{}^g(h_1\ldots h_n).
\end{align*}
\end{proof}
\begin{remark}
The above proposition is an analogue of \cite[Propositions 4.6, 4.7]{LiTr} for conjugation map in group cohomology. Alternative proofs for \cite[Propositions 4.6, 4.7]{LiTr} can be given using the explicit descriptions from \ref{resH}, \ref{trH} and \ref{diagH}. Although the elegant way to define the above transfer maps are given in \cite{LiTr} it is our strong belief that the explicit approach given in this section could prove to be useful in further applications.
\end{remark}

\section{Transfer maps between cohomology algebras of block algebras of finite groups}\label{sec4}
First we give the reciprocity laws in the following proposition.

\begin{proposition}\label{propreciprtransf} Let $b,c$ be blocks as above such that the Conjecture is true for $b$. Let $[\zeta]\in \Hc^*(G,b,P_{\gamma}), [\tau]\in\Hc^*(H,c,Q_{\delta})$. The following statements hold.
\begin{itemize}
\item[(i)] $\tr^b_c(\res^b_c([\zeta])[\tau])=[\zeta]\tr^b_c([\tau])$;
\item[(ii)] $\tr^b_c([\tau] \res^b_c([\zeta]))=\tr^b_c([\tau])[\zeta]$.
\end{itemize}
\end{proposition}
\begin{proof} Since (ii) is analogue to (i) we only prove (i).
\begin{align*}
\tr^b_c(\res^b_c([\zeta])[\tau])&=\sum_{g\in Y_{G,b}}\tr^P_{P\cap ^gP}(\res^{^gP}_{P\cap ^gP}{}^g(\tr^P_Q(\res^P_Q([\zeta])[\tau])))\\&=\sum_{g\in Y_{G,b}}\tr^P_{P\cap ^gP}(\res^{^gP}_{P\cap ^gP}{}^g([\zeta]\tr^P_Q([\tau]))\\&=\sum_{g\in Y_{G,b}}\tr^P_{P\cap ^gP}(\res^{^gP}_{P\cap ^gP}(^g[\zeta])\res^{^gP}_{P\cap ^gP}(^g\tr^P_Q([\tau])))\\&=\sum_{g\in Y_{G,b}}\tr^P_{P\cap ^gP}(\res^{P}_{P\cap ^gP}([\zeta])\res^{^gP}_{P\cap ^gP}(^g\tr^P_Q([\tau])))\\&=\sum_{g\in Y_{G,b}}[\zeta]\tr^P_{P\cap ^gP}(\res^{^gP}_{P\cap ^gP}(^g\tr^P_Q([\tau])))\\&=[\zeta]\tr^b_c([\tau]),
\end{align*}
where the fourth equality is true since $[\zeta]\in\Hc^*(G,b,P_{\gamma})$.
\end{proof}
From Proposition \ref{propreciprtransf} we obtain the transitivity laws when a supplementary condition is satisfied.
\begin{proposition}\label{proptranslaws}We keep the above assumptions.
\begin{itemize}
\item[(a)] If $Q<P$ then $\tr^b_c\circ \res^b_c=0$;
\item[(b)] If $Q=P$ then $\res^b_c$ is the inclusion map and $\tr^b_c(\res^b_c([\zeta]))=\frac{\dim_k(ikGi)}{  \mid P\mid}[\zeta]$ for any $[\zeta]\in\Hc^*(G,b,P_{\gamma})$. Moreover in this case if the Conjecture is also true for $c$ and $\Ker~t_{Y_{H,c,P_{\delta}}}\subseteq \Ker~t_{Y_{G,b,P_{\gamma}}}$ the following diagram is commutative
$$\xymatrix{
\Hc^*(H,c,P_{\delta}) \ar@{->>}[rr]^{\tr_c^b}&&\Hc^*(G,b,P_{\gamma}) \\
& \Hc^*(P,k) \ar@{->>}[ul]^{t_{Y_{H,c,P_{\delta}}}} \ar@{->>}[ur]_{t_{Y_{G,b,P_{\gamma}}}}}.$$
\end{itemize}
\end{proposition}
\begin{proof}
Let $[1]$ be the unity of $\Hc^*(G,b,P_{\gamma})$ and $\Hc^*(H,c,Q_{\delta})$. It is clear from Definition \ref{defntrbc} that
$$\tr_c^b([1])=\sum_{g\in Y_{G,b,P_{\gamma}}}\tr^P_{P\cap^gP}\res^{^gP}_{P\cap^gP}{}^g(\tr^P_Q([1]))$$
is equal to $0$ if $Q<P$ or it is equal to
$$\sum_{g\in Y_{G,b,P_{\gamma}}}|P:P\cap{}^gP|[1]=\frac{\dim_k ikGi}{|P|}[1],$$
if $Q=P$. Then, by Proposition \ref{propreciprtransf}, (ii) we have
$$\tr_c^b(\res_c^b([\zeta]))=\tr_c^b([1])[\zeta]$$
which is equal to $0$ if $Q<P$ or it is equal to $\frac{\dim_k ikGi}{|P|}[\zeta]$ if $Q=P$.

Since the Conjecture holds for both $b$ and $c$ the next two short exact sequences split as sequences of $\Hc^*(G,b,P_{\gamma})$-modules, respectively of $\Hc^*(H,c,P_{\delta})$-modules; see \cite[Remark 1]{Sas}
$$\xymatrix{0 \ar[r]&\Ker t_{Y_{G,b,P_{\gamma}}}\ar@{^{(}->}[r]
&\Hc^*(P,k)\ar@{->>}[rr]^{t_{Y_{G,b,P_{\gamma}}}}&&\Hc^*(G,b,P_{\gamma})\ar[r] &0},$$
    $$\xymatrix{0 \ar[r]&\Ker t_{Y_{H,c,P_{\delta}}}\ar@{^{(}->}[r]
&\Hc^*(P,k)\ar@{->>}[rr]^{t_{Y_{H,c,P_{\delta}}}}&&\Hc^*(H,c,P_{\delta})\ar[r] &0},$$
hence $$\Hc^*(P,k)=\Ker~t_{Y_{H,c,P_{\delta}}}\bigoplus \Hc^*(H,c,P_{\delta})=\Ker~ t_{Y_{G,b,P_{\gamma}}}\bigoplus \Hc^*(G,b,P_{\gamma})$$ as graded $k$-vector spaces. Now we take $$[\zeta]=[\zeta_1]+[\zeta_2]\in\Ker~t_{Y_{H,c,P_{\delta}}}\bigoplus \Hc^*(H,c,P_{\delta})$$ and we obtain
$$\tr_c^b(t_{Y_{H,c,P_{\delta}}}([\zeta_1]+[\zeta_2]))=\tr_c^b([\zeta_2])=t_{Y_{G,b,P_\gamma}}([\zeta_2]).$$
Since $\Ker~t_{Y_{H,c,P_{\delta}}}\subseteq \Ker~t_{Y_{G,b,P_{\gamma}}}$ we get
$$t_{Y_{G,b,P_{\gamma}}}([\zeta_1]+[\zeta_2])=t_{Y_{G,b,P_{\gamma}}}([\zeta_2]).$$

\end{proof}
\begin{remark} From Proposition \ref{proptranslaws}, (ii) we can notice the fact that the name "transfer map" in Definition \ref{defntrbc} is more suitable for blocks $b,c$ for which the Conjecture of Sasaki is true and for which we have $\Ker t_{Y_{H,c,P_{\delta}}}\subseteq \Ker t_{Y_{G,b,P_{\gamma}}}$.
\end{remark}
We end with the proof of Theorem \ref{thmcomptran}.

\textbf{Proof of Theorem \ref{thmcomptran}}
 First we prove the commutativity of the next diagram
 \begin{displaymath}
 \xymatrix{\Hc^*(H,c,P_{\delta})\ar[rr]^{ \delta_P}\ar[d]^{\tr^b_c} && \HH^*(kP) \ar[d]^{t_{ikGi}} \\
                                                     \Hc^*(G,b,P_{\gamma})\ar[rr]^{\delta_P}
                                                     &&\HH^*(kP).
                                            }.
\end{displaymath} Let $[\zeta]\in \Hc^*(H,c,P_{\delta})$. The following equalities hold
\begin{align*}
(\delta_P\circ\tr^b_c)([\zeta])&=\sum_{g\in Y_{G,b,P_{\gamma}}}(\delta_P\circ \tr^P_{P\cap^gP})(\res^{^gP}_{P\cap ^gP}({}^g[\zeta])\\&=\sum_{g\in Y_{G,b,P_{\gamma}}}(t_{_PkP_{P\cap {}^gP}}\circ\delta_P\circ\res^{^gP}_{P\cap ^gP})({}^g[\zeta])\\&=\sum_{g\in Y_{G,b,P_{\gamma}}}(t_{_PkP_{P\cap {}^gP}}\circ t_{_{P\cap ^gP}k[^gP]_{^gP}}\circ\delta_P)({}^g[\zeta])\\&=\sum_{g\in Y_{G,b,P_{\gamma}}}t_{_PkP\otimes_{k[P\cap ^gP]}k[^gP]_{^gP}}(c_{g,P}(\delta_{P}([\zeta])))\\&=\sum_{g\in Y_{G,b,P_{\gamma}}}t_{_PkP\otimes_{k[P\cap ^gP]}k[^gP]\otimes_{k[^gP]}k[gP]_P}(\delta_P([\zeta]))\\&=\sum_{g\in Y_{G,b,P_{\gamma}}}t_{_Pk[PgP]_P}(\delta_P([\zeta]))\\&=(t_{ikGi}\circ\delta_P)([\zeta]),
\end{align*}
where the fourth and sixth equality follow from Proposition \ref{propcompconjdelta} and Lemma \ref{lemkgh}; the other equalities are straightforward applications of properties of transfer maps from \cite{LiTr}. Since by \cite[Theorem 1]{Sas} $$\delta_P:\Hc^*(H,c,P_{\delta})\rightarrow \HH^*_{B^*}(kP)\cap \delta_P(\Hc^*(P,k))$$ is an isomorphism and $$\mathcal{F}_{(P,f_P)}(H,c)\subseteq \mathcal{F}_{(P,e_P)}(G,b),$$ the above diagram and \cite[Theorem 2]{Sas} assure us that $$\Ima(T)\subseteq \HH^*_{A^*\otimes_{kGb}L\otimes_{kHc}B}(kP).$$ Consequently, the above diagram gives us
\begin{displaymath}
 \xymatrix{\Hc^*(H,c,P_{\delta})~~\ar@{>->>}[r]^{\delta_P~~~~~~~~~~~~~~}\ar[d]^{\tr_c^b} &\HH^*_{B^*}(kP)\cap \delta_P(\Hc^*(P,k)) \ar[r]^{T}&\HH^*_{B^*\otimes_{kHc}L^*\otimes_{kGb}A}(kP)\ar@{=}[d] \\
          \Hc^*(G,b,P_{\gamma})\ar@{>->}[rr]^{\delta_P}&&\HH^*_{A^*\otimes_{kGb}L\otimes_{kHc}B}(kP)}.
\end{displaymath}

Furthermore, we glue this diagram with the commutative diagrams  from the center of \cite[Theorem 4, (6)]{Sas}
 \begin{displaymath}
 \xymatrix{\HH^*_{B^*\otimes_{kHc}L^*\otimes_{kGb}A}(kP)~\ar@{>->>}[r]^{R_B}\ar@{=}[d] &\HH^*_{L^*}(kHc)\cap \HH^*_B(kHc)\ar@{^{(}->}[r]\ar@{>->>}[d]^{R_L} & \HH^*_{L^*}(kHc)~~\ar@{>->>}[d]^{R_L} \\
                                                     \HH^*_{A^*\otimes_{kGb}L\otimes_{kHc}B}(kP)~~\ar@{>->>}[r]^{R_A}&\HH^*_{L\otimes_{kHc}B}(kGb)\ar@{^{(}->}[r]
                                                     &\HH^*_L(kGb)
                                            }
\end{displaymath}
to obtain the conclusion.


\begin{thebibliography}{00}

\bibitem{BLO}
C. Broto, R. Levi and B. Oliver, The homotopy theory of fusion systems, {\it J. Amer. Math. Soc.} {\bf16} (2003), 779--856.
\bibitem{CoTo} T. Cocone\c t and C.-C. Todea, Hochschild cohomology of fully group-graded algebras as Mackey functor, {\it Annals of Tiberiu Popovici Seminar} {\bf 12}  (2014), 17--22.
\bibitem{LiZh}
S. K\"{o}nig, Y. Liu and G. Zhou, Transfer maps in Hochschild (co)homology and applications to stable and derived invariants and to the Auslander-Reiten conjecture, {\it Trans. Amer. Math. Soc.} {\bf 364} (2012), 195--232.
\bibitem{LiTr}
M. Linckelmann, Transfer in Hochschild cohomology of blocks of finite groups, {\it Alg. Represent. Theory} {\bf 2} (1999), 107--135.
\bibitem{PuBook}
L. Puig, {\it On the Local Structure of Morita and Rickard Equivalences between Brauer Blocks}, Progress in Mathematics (Boston, Mass.), Vol. 178, Birkh\"{a}user, Basel, 1999.
\bibitem{PuHyp}
L. Puig, The hyperfocal subalgebra of a block, {\it Invent. Math.} {\bf 141} (2000), 365--397.
\bibitem{PuNilext}
L. Puig, Nilpotent extensions of blocks, {\it Math. Zeitschrift.} {\bf 269} (2011), 115--136.
\bibitem{SasII} H. Sasaki,
Stable elements in cohomology algebras,{\it Cohomology Theory of Finite Groups and Related Topics},	http://hdl.handle.net/2433/141327, {\bf 1679} (2010), 1--6.
\bibitem{Sas} H. Sasaki,
 Cohomology of block ideals of finite group algebras and stable elements, {\it Alg. Repres. Theory} {\bf 16} (2013), 1039-1049.
\bibitem{SiWh} S. Siegel and S. Witherspoon,
The Hochschild chomology ring of the group algebra, {\it Proc. London Math. Soc.} {\bf 79} (1999), 131--157.
\bibitem{TH}
J. Th\'{e}venaz, {\it G-Algebras and Modular Representation
Theory}, Clarendon Press, Oxford, 1995.
\bibitem{Wat} A.  Watanabe,
On blocks of finite groups with central hyperfocal subgroups, {\it J. Algebra} {\bf 368} (2012), 358–-375.
\bibitem{We} E. Weiss,
{\it Cohomology of groups}, Academic Press, New York , 1969.

\end{thebibliography}

\end{document}